\documentclass[runningheads]{amsart}

\usepackage{amsmath}
\usepackage{amsfonts}
\usepackage{epsfig}
\usepackage{graphics}

\def\meas{{\rm meas}\, }

\begin{document}
\newtheorem{theorem}{Theorem}[section]
\newtheorem{lemma}[theorem]{Lemma}
\theoremstyle{definition}
\newtheorem{definition}[theorem]{Definition}
\newtheorem{exercise}[theorem]{Exercise}
\newtheorem{example}[theorem]{Example}
\newtheorem{remark}[theorem]{Remark}
\title{The minimum angle condition for $d$-simplices}


\author{Sergey Korotov \and Jon Eivind Vatne}


\address{
Department of Computing, Mathematics and Physics, Western Norway University of Applied Sciences,
Post Box 7030, N-5020 Bergen, Norway, {\tt sergey.korotov@hvl.no, jon.eivind.vatne@hvl.no}
}

\maketitle

\begin{abstract}
  In this note we present a natural generalization of the minimum angle condition,
  commonly used in the finite element analysis for planar triangulations, to the case
  of simplicial meshes in any space dimension.
  The equivalence of this condition with some  other mesh regularity conditions
  is proved.
\end{abstract}

\section{Introduction}
\label{chapter-intro}

Various regularity properties (usually prescribed in terms of geometric characteristics)
are required from the meshes/partitions of space domains in order to guarantee suitable
interpolation and convergence results.  
In this paper we present and analyse a natural higher-dimensional
analogue of the so-called minimum angle condition, commonly imposed on planar triangulations
in the finite element analysis,
which roughly speaking forbids the mesh elements to shrink.

First, we recall the  basic results on the topic.
Let $\mathcal F = \{\mathcal T_h\}_{h \to 0}$ be a family of conforming triangulations
$\mathcal T_h$ of a bounded polygonal domain.
In 1968
the following {\it minimum angle condition} was proposed \cite{Zla,Zen}:
there exists a constant $\alpha_0 > 0$
such that for any triangulation $\mathcal T_h \in \mathcal F$ and any triangle $T \in \mathcal T_h$
the bound
\begin{equation}
  \alpha_T \ge \alpha_0,
  \label{eq-lower-bound}
\end{equation}
holds, where $\alpha_T$ is the minimum angle of $T$.

Later, many other (geometric) conditions on triangulations equivalent to (\ref{eq-lower-bound})
have been proposed, see e.g. \cite{BraKorKri-2008,BraKorKri-2009,BraKorKri-Zlamal}
and references therein. Moreover, some higher-dimensional analogues of  (\ref{eq-lower-bound})
were proposed \cite{BraKorKri-Zlamal,Cia}. Under all these conditions  various interpolation
and finite element convergence estimates
are usually derived, we refer to \cite{Cia} as one of the basic sources in this respect.
In what follows, we will mainly concentrate on mesh regularity
conditions based on estimates of dihedral angles.

Thus, in 2008,
in \cite{BraKorKri-2008}, condition (\ref{eq-lower-bound})
was generalized 
to tetrahedral elements as follows: there exists a constant $\alpha_0 > 0$
such that for any conforming tetrahedralization $\mathcal T_h \in \mathcal F$ and any
tetrahedron $T \in \mathcal T_h$ one has 
\begin{equation}
  \alpha_{\rm D}^T \ge \alpha_0 \quad \text{and}   \quad   \alpha_{\rm F}^T \ge \alpha_0, 
  \label{eq-lower-bound-3D}
\end{equation}
where $\alpha_{\rm D}^T$ is the minimum of values of dihedral angles between faces of $T$ and
$\alpha_{\rm F}^T$ is the minimum angle in all four triangular faces of $T$. This condition was
also proved to be equivalent with many other regularity conditions on tetrahedral meshes, see
\cite{BraKorKri-2008}.

However, the next step of proposing the higher-dimensional condition in terms of
dihedral angles has not been done so far. In this paper we address this issue to.
Namely, we recall the minimum angle-type condition based on the concept of
$d$-sine of F. Eriksson (see \cite{Eri}) as presented   in the paper \cite{BraKorKri-Zlamal},
then propose a natural higher-dimesional analogue of (\ref{eq-lower-bound}) and (\ref{eq-lower-bound-3D}),
and prove the equivalence of these two regularity conditions in any dimension.

\section{Minimum angle conditions in higher dimensions}
\label{chapter-denotation}

A $d$-simplex $S$ in ${\bf R}^d, d \in \{1, 2, 3, \dots \}$,
is the convex hull of $d+1$ vertices $A_0, A_1, \dots, A_{d}$ that do not
belong to the same $(d-1)$-dimensional hyperplane, i.e.,
$$
S =  {\rm conv}  \{A_0, A_1, \dots , A_{d}\}.
$$
Further,
$$
F_i =  {\rm conv} \{A_0, \dots , A_{i-1}, A_{i+1}, \dots ,A_{d}\}
$$ 
is the $(d-1)$-facet of $S$ opposite to the vertex $A_i$ for $i \in \{0, \dots, d\}$.

For $d \ge 2$ the dihedral angle $\beta_{ij}$ between two $(d-1)$-facets $F_i$ and $F_j$ of $S$
is defined by means of the inner product of their outward unit normals $n_i$ and $n_j$
\begin{equation}
\cos \beta_{ij} = - n_i  \cdot n_j.
\label{dihedral-angle}
\end{equation}

In 1978, Eriksson introduced a generalization of the sine function to an arbitrary
$d$-dimensional spatial angle, see \cite[p.~74]{Eri}.

\begin{definition}
  Let $\hat A_i$ be the angle at the vertex $A_i$ of the simplex $S$. Then  $d$-sine of the angle
  $\hat A_i$ for $d>1$ is given by
\begin{equation}
\sin_d (\hat A_{i}  | A_0 A_1 \dots A_{d}) 
= \frac{ d^{d-1} \, (\meas_d S)^{d-1} }{(d-1) ! \, {\bf \Pi}_{j = 0, j \neq i}^{d} \meas_{d-1} F_j }. 
\label{d-sine}
\end{equation}
\end{definition}

\begin{remark}
  The $d$-sine is really a generalization of the classical sine function. In order to
  see that, set $d=2$ and consider an arbitrary triangle $A_{0} A_{1} A_{2}$. Let $\hat A_{0}$
be its angle at the vertex $A_{0}$. Then, obviously, 
\begin{equation}
\meas_2 (A_0 A_1 A_2) = {1\over 2} |A_0 A_1| |A_0 A_2| \sin \hat A_0,
\label{2d-area-triangle}
\end{equation}
which implies
\begin{equation}
  \sin \hat A_0 = \sin_2 (\hat A_0 | A_0 A_1 A_2). 
\label{eq-compare}
\end{equation}
\end{remark}

\begin{definition} 
  A family $\mathcal F = \{\mathcal T_h \}_{h \to 0}$  of partitions of a polytope
   $\overline \Omega \subset {\bf R}^d$ into 
$d$-simplices is said to satisfy the {\it generalized minimum angle condition} if
there exists $C > 0$ such that for any $\mathcal T_h \in \mathcal F$ and any
$S = {\rm conv} \{ A_0, \dots , A_{d} \} \in \mathcal T_h$ 
one has 
\begin{equation}
\forall \ i \in \{0, 1, \dots , d \}
\qquad
\sin_d ( \hat A_i  | A_{0} A_{1} \dots A_{d})  \ge C > 0.
\label{eq_dsine_bounded}
\end{equation}
\label{conditions-4-any-dimension}
\end{definition} 

Now we present a natural generalization of conditions  (\ref{eq-lower-bound}) and
(\ref{eq-lower-bound-3D}) for simplicial meshes in any dimension.

\begin{definition}
    A family $\mathcal F  = \{\mathcal T_h \}_{h \to 0}$ of partitions of a polytope
  $\overline \Omega \subset {\bf R}^d$ into 
$d$-simplices is said to satisfy the {\it $d$-dimensional  minimum angle condition} if
    there exists a constant $\alpha_0 > 0$ such that for $\mathcal T_h \in \mathcal F$
    and any simplex $S \in \mathcal T_h$  and any subsimplex $S'\subseteq S$ with vertex
    set contained in the vertex set of $S$, the minimum dihedral angle in $S'$ is not less
    than $\alpha_0$. 
\label{def-minimum-general}
\end{definition}

\begin{remark}
  Obviously, the above condition coincides with (\ref{eq-lower-bound}) for the case
  $d=2$, and with (\ref{eq-lower-bound-3D}) -- for $d=3$, and by nature it has a form
  of limiting all the dihedral angles from below.
  In case $S$ is a triangle, the dihedral angles are the ordinary angles.
  Definitions~\ref{conditions-4-any-dimension} and \ref{def-minimum-general} thus
  express exactly the same condition in dimension two.
  \label{remark-low-dimension}
\end{remark}

\begin{lemma}
For a $d$-simplex we observe that
 \begin{equation}
\sin_d (\hat A_{i}  | A_0 A_1 \dots A_{d}) 
= \sin_{d-1} (\hat A_{i}  | A_0 A_1 \dots A_{d-1})\prod_{j=0, j\neq i}^{d-1}\sin(\beta_j), 
\label{d-sine-product}
\end{equation}
 where $\beta_j$ is the dihedral angle between the facet omitting $A_j$ and the facet
 omitting $A_d$.
\label{lemma-sines-tetrahedron-ABCD}
\end{lemma}

For the proof see \cite[p. 74--76]{Eri}.

\begin{theorem}
  The $d$-dimensional minimum angle condition presented in  Definition~\ref{def-minimum-general} and
  the generalized minimum angle condition from Definition~\ref{conditions-4-any-dimension} are equivalent.
\end{theorem}
%
\begin{proof}
The proof is by induction on the dimension $d$.
For $d=2$ the two conditions are in fact the same, as pointed out in Remark~\ref{remark-low-dimension}.
Assume by induction that the theorem is true for all dimensions smaller than $d$.\\

Assume first that the condition in Definition~\ref{conditions-4-any-dimension} is satisfied, so that there exists a constant $C>0$ as in \eqref{eq_dsine_bounded}.
Using \eqref{d-sine-product}, we get that
 \begin{equation}
C\leq \sin_d (\hat A_{i}  | A_0 A_1 \dots A_{d}) 
= \sin_{d-1} (\hat A_{i}  | A_0 A_1 \dots A_{d-1})\prod_{j=0, j\neq i}^{d-1}\sin(\beta_j) .
\end{equation}
All the factors on the right are bounded from above by $1$, and therefore each of them must be larger than or equal to $C$.
In this formula, all the dihedral angles involve the face opposite the vertex $A_d$.
However, by reordering the vertices, we get that any dihedral angle $\beta$ satisfies $\sin\beta\geq C$.
Since $\sin_{d-1} (\hat A_{i}  | A_0 A_1 \dots A_{d-1})\geq C$ we can use the induction hypothesis to conclude that all dihedral angles in simplices of strictly smaller dimension than $d$ are not smaller than a constant $\alpha_0'$, say.
We can then let $\alpha_0=\min\{\alpha_0',\arcsin C\}$ and see that the condition in Definition~\ref{def-minimum-general} is satisfied.\\

Assume next that the condition in Definition~\ref{def-minimum-general} holds, so that $\alpha_0>0$ is a lower bound for all the dihedral angles considered in any dimension.
Since any subsimplex of a facet is also a subsimplex of the whole simplex, we get by induction that there is a constant $C'$ such that
\begin{equation}
\sin_{d-1} (\hat A_{i}  | A_0 A_1 \dots A_{d-1})\geq C'.
\end{equation}
By the product formula \eqref{d-sine-product} we then get
 \begin{equation*}
\sin_d (\hat A_{i}  | A_0 A_1 \dots A_{d}) 
= \sin_{d-1} (\hat A_{i}  | A_0 A_1 \dots A_{d-1})\prod_{j=0, j\neq i}^{d-1}\sin(\beta_j)
 \geq C'\prod_{j=0, j\neq i}^{d-1}\sin(\beta_j).
\end{equation*}
By the minimal angle condition, we know that $\beta_j>\alpha_0$.
We will show that there also exists an upper bound $\gamma_0<\pi$ such that $\beta_j\leq \gamma_0$.
Then $\sin\beta_j\geq \min(\sin \alpha_0,\sin\gamma_0)$ and
 \begin{equation*}
\sin_d (\hat A_{i}  | A_0 A_1 \dots A_{d}) 
 \geq C'\prod_{j=0, j\neq i}^{d-1}\sin(\beta_j)\geq C'\big(\min(\sin \alpha_0,\sin\gamma_0)\big)^{d-1}.
\end{equation*}
We then get the theorem by setting
\begin{equation}
C=C'\big(\min(\sin \alpha_0,\sin\gamma_0)\big)^{d-1}.
\end{equation}
Assume to get a contradiction that there is no upper bound for the dihedral angles smaller than $\pi$.
There is then a sequence $S_h$ of simplices with a dihedral angle tending to $\pi$.
In the limit, the sequence must degenerate into a set spanning a proper subspace $H\subset\mathbb{R}^d$.
By the induction hypothesis, any subsimplex of dimension $k$ satisfy that all $\sin_{k}$ at any vertex is bounded from below.
It follows that any subsimplex will not give a degenerate limit.
In particular, the limit of any facet will span a subspace of dimension $d-1$, namely $H$.
Therefore the outward normals of the facets will tend to normals of $H$.
Now the dihedral angle between two facets cannot tend to zero by the minimal angle condition.
Therefore all the outward normals must tend to normals of $H$ pointing in the same direction,
which again implies that {\em all} the dihedral angles tend to $\pi$.
This is clearly false (e.g. it violates the upper bound for the sum of the dihedral angles,
see e.g. \cite{Gad-1956}), and we have established our contradiction.
\end{proof}

\section*{Final remarks}  

\begin{itemize}

  \item
Generation of meshes under satisfying the minimum angle conditions is presented and discussed
e.g. in \cite{HanKorKri-2010,HanKorKri-2014}.

\medskip

  \item
The proposed $d$-dimensional  minimum angle condition is also equivalent to the classical
inscribed ball condition of Ciarlet \cite{Cia} due to the equivalence results of \cite{BraKorKri-2009}.
(Several other less known equivalent regularity conditions can be found in the same reference.)

\end{itemize}

%

%
%
%

\ifx\undefined\bysame
\newcommand{\bysame}{\leavevmode\hbox to3em{\hrulefill}\,}
\fi

\end{document}